\documentclass{amsart}
\usepackage{amssymb}
\usepackage{verbatim,tikz}
\usepackage{mathrsfs, amsthm}
\usepackage{dsfont}
\usepackage[all]{xy}
\usepackage{eurosym}

\newcommand\II{\text{II}}
\newcommand\I{\text{I}}
\newcommand\IR{\text{I,R}}
\newcommand\IL{\text{I,L}}

\newcommand\mr{\operatorname{mr}}
\newcommand\LT{{}^LT}

\newcommand\Pl{\text{Pl}}
\newcommand\ilog{\operatorname{ilog}}
\newcommand\Mm{M_{\mr}}
\newcommand\tq{\frac34}

\newcommand\mrurl[1]{MR. \url{http://www.ams.org/mathscinet-getitem?mr=#1}}
\newcommand\aurl[1]{\url{file:///data/ms/archive/#1}}
\newcommand\durl[1]{doi. \url{http:///dx.doi.org/#1}}
\newcommand\purl[1]{\url{file:///data/ms/print/#1}}
\newcommand\surl[1]{\url{file:///data/ms/screen/#1}}
\newcommand\zurl[1]{Zbl. \url{http://www.emis.de/zmath-item?#1}}
\newcommand\Aurl[1]{Arxiv. \url{http://arxiv.org/abs/#1}}
\newcommand\datver[1]{\def\datverp%
 {\par\boxed{\boxed{\text{#1; Run: \today}}}}}

\newcommand\boxb[1]{\square_b}

\numberwithin{equation}{section}

\newcommand\paperbody%
        {}

\newtheorem{lemma}{Lemma}
\newtheorem{proposition}{Proposition}

\newtheorem*{theorem*}{Theorem}
\newtheorem{non-theorem}{Non-Theorem}

\theoremstyle{remark}

\newcommand\bV{\mathcal{V}_{\operatorname{b}}}

\newcommand\bo{\operatorname{b}}

\newcommand\bT{{}^{\bo}T}

\newcommand\cFTs{{}^{\Phi}\overline{T}\kern-1pt{}^*}

\newcommand\fib{\operatorname{fib}}

\newcommand\Nul{\operatorname{Nul}}

\newcommand\dbar{\overline{\pa}}

\newcommand\cA{\mathcal{A}}

\newcommand\cI{\mathcal{I}}

\newcommand\cP{\mathcal{P}}

\newcommand\cU{\mathcal{U}}

\newcommand\cW{\mathcal{W}}

\newcommand\bbC{\mathbb C}
\newcommand\bbD{\mathbb D}

\newcommand\bbN{\mathbb N}

\newcommand\bbR{\mathbb R}
\newcommand\bbS{\mathbb S}

\newcommand\bbZ{\mathbb Z}

\newcommand\CIc{{\mathcal{C}}^{\infty}_\text{c}}

\newcommand\CI{{\mathcal{C}}^{\infty}}

\newcommand\CmI{{\mathcal{C}}^{-\infty}}

\newcommand\Diff[1]{\operatorname{Diff}^{#1}}

\newcommand\Diffb[1]{\operatorname{Diff}^{#1}_{\text{b}}}

\newcommand\cFNs{{}^{\Phi}\overline N\kern-1pt{}^*}

\newcommand\reg{\operatorname{reg}}

\newcommand\Lap{\varDelta}

\newcommand\ci{${\mathcal{C}}^\infty$}

\newcommand\dCIc{\dot{\mathcal{C}}^{\infty}_c}
\newcommand\dCI{\dot{\mathcal{C}}^{\infty}}

\newcommand\ha{\frac{1}{2}}

\newcommand\pa{\partial}

\renewcommand\Re{\operatorname{Re}}
\renewcommand\Im{\operatorname{Im}}

\newcommand\Mand{\text{ and }}

\newcommand\Mst{\text{ s.t. }}

\newcommand\Mwith{\text{ with }}

\datver{0.1B; Revised: 13-10-2014}
\begin{document}
\title[Metrics and Lefschetz fibrations]
{Resolution of the canonical fiber metrics for a Lefschetz fibration}

\author{Richard Melrose}
\author{Xuwen Zhu}
\address{Department of Mathematics, Massachusetts Institute of Technology}
\email{rbm@math.mit.edu}
\address{Department of Mathematics, Massachusetts Institute of Technology}
\email{xuwenz@math.mit.edu}

\begin{abstract} We consider the family of constant curvature fiber metrics
  for a Lefschetz fibration with regular fibers of genus greater than
  one. A result of Obitsu and Wolpert is refined by showing that on an
  appropriate resolution of the total space, constructed by iterated
  blow-up, this family is log-smooth, i.e.\ polyhomogeneous with integral
  powers but possible multiplicities, at the preimage of the singular
  fibers in terms of parameters of size comparable to the logarithm of the
  length of the shrinking geodesic.
\end{abstract}

\maketitle


\section*{Introduction}

In the setting of complex surfaces, a Lefschetz fibration is a holomorphic
map to a curve, generalizing an elliptic fibration in that it has only a
finite number of singular points near which it is holomorphically reducible
to normal crossing. Donaldson \cite{MR1648081} showed that a
four-dimensional simply-connected compact symplectic manifold, possibly
after stabilization by a finite number of blow-ups, admits a Lefschetz
fibration, in an appropriately generalized sense, over the sphere;
Gompf~\cite{gompf1995new} showed the converse.  The reader is referred to
the book of Gompf and Stipsicz \cite{robert19994} for a description of the
important role played by Lefschetz fibrations in the general theory of
4-manifolds.

To cover these cases we consider a compact connected almost-complex
4-manifold $M$ and a smooth map, with complex fibers, to a Riemann surface
$Z$
\begin{equation}
\xymatrix{
M\ar[r]^{\psi}&Z.
}
\label{MetLef.2}\end{equation}
We then require that this map be pseudo-holomorphic, have surjective
differential outside a finite set $F\subset M,$ on which $\psi$ is
injective, so $\psi:F\longleftrightarrow S\subset M,$ and near each of these
singular points be reducible to the normal crossing, or plumbing variety,
model \eqref{MetResMod.3} below.

A curve of genus $g$ with $b$ punctures is stable if its automorphism group
is finite, which is the case when $3g-3+b>0.$ In this paper we discuss Lefschetz
fibrations with regular fibers having genus $g>1$ and hence stable. All
fibers carry a unique metric of curvature $-1,$ for the singular fibers
with cusp points replacing the nodes. In view of uniqueness and stability,
these metrics necessarily vary smoothly near a regular fiber. We discuss
here the precise uniform behavior of this family of metrics near the
singular fibers, showing that in terms of appropriate (logarithmic)
resolutions, of both the total and parameter spaces, to manifolds with
corners the resulting fiber metric is polyhomogeneous and more particularly
log-smooth, i.e.\ essentially smooth except for the appearance of
logarithmic terms in the expansions at boundary surfaces. This refines a
result of Obitsu and Wolpert \cite{MR2399166} who gave the first two terms
in the expansion. In a forthcoming paper the universal case of the
Deligne-Mumford compactification of the moduli space of Riemann surfaces,
also treated by Obitsu and Wolpert, will be discussed.

The local model for degeneration for the complex structure on a Riemann
surface to a surface with a node is the `plumbing variety' with its
projection to the parameter space. We add boundaries, away from the
singularity at the origin, to make this into a manifold with corners:
\begin{equation}
\begin{gathered}
P=\{(z,w)\in\bbC^2;\ \exists\ t\in\bbC, \ zw=t,\ |z|\le\tq,\ |w|\le\tq,\ |t|\le\ha\}\\
P\overset{\phi}\longrightarrow\bbD_{\ha}=\{t\in\bbC;|t|\le\ha\}.
\end{gathered}
\label{MetResMod.3}\end{equation}
Thus near each point of $F$ we require that $\psi$ can be reduced to $\phi$
in (almost) holomorphic coordinates in $M$ and $Z.$ 

A (real) manifold with corners $M$ has a principal ideal
$\cI_F\subset\CI(M)$ corresponding to each boundary hypersurface (by
assumption embedded and connected) generated by a boundary defining
function $\rho_F\ge0$ with $F=\{\rho _F=0\}$ and $d\rho _F\not=0$ on $F.$ A
smooth map between manifolds with corners $f:M\longrightarrow Y$ is an
interior b-map if each of these ideals on $Y$ pulls back to non-trivial
finite products of the corresponding ideals on $M,$ it is b-normal if there
is no common factor in these product decompositions -- this is always the
case here since the range space is a manifold with boundary. Such a map is
a b-fibration if in addition every smooth vector field tangent to all
boundaries on $Y$ is locally (and hence globally) $f$-related to such a
vector field on $M;$ it is then surjective. There is a slightly weaker
notion than a manifold with corners, a \emph{tied} manifold, which has the
same local structure but in which the boundary hypersurfaces need not be
embedded, meaning that transversal self-intersection is allowed. This
arises below, although not in any essential way. There is still a principal
ideal associated to each boundary hypersurface and the notions above carry over.

The assumptions above mean that each singular fiber of $\psi$ has one
singular point at which it has a normal crossing in the (almost) complex
sense as a subvariety of $M.$ The first step in the resolution is the blow
up, in the real sense, of the singular fibers; this is well-defined in view
of the transversality of the self-instersection but results in a tied
manifold since the boundary faces are not globally embedded. The second
step is to replace the \ci\ structure by its logarithmic weakening,
i.e.\ replacing each (local) boundary defining function $x$ by 
\begin{equation*}
\ilog x=(\log x^{-1})^{-1}.
\label{MetLef.26}\end{equation*}
This gives a new tied manifold mapping smoothly to the previous one by a
homeomorphism. These two steps can be thought of in combination as the
`logarithmic blow up' of the singular fibers. The final step is to blow up
the corners, of codimension two, in the preimages of the singular
fibers. This results in a manifold with corners, $M_{\mr},$ with the two
boundary hypersurfaces denoted $B_{\I},$ resolving the singular fiber, and
$B_{\II}$ arising at the final stage of the resolution. The parameter space
$Z$ is similarly resolved to a manifold with corners by the logarithmic
blow up of each of the singular points.

It is shown below that the Lefschetz fibration lifts to a smooth map 
\begin{equation}
\xymatrix{
\Mm \ar[r]^{\psi_{\mr}}&Z_{\mr}
}
\label{MetLef.3}\end{equation}
which is a b-fibration. In particular it follows from this that smooth
vector fields on $M_{\mr}$ which are tangent to all boundaries and to the
fibers of $\psi_{\mr}$ form the sections of a smooth vector subbundle of
$\bT M_{\mr}$ of rank two. The boundary hypersurface $B_{\II}$ has a preferred class of
boundary defining functions, an element of which is denoted $\rho _{\II},$
arising from the logarithmic nature of the resolution, and this allows a
Lie algebra of vector fields to be defined by
\begin{equation}
V\in\CI(M_{\mr};\bT M_{\mr}),\ V\psi^*\CI(Z_{\mr})=0,\ V\rho _{\II}\in\rho _{\II}^2\CI(M_{\mr}).
\label{MetLef.4}\end{equation}
The possibly singular vector fields of the form $\rho _{\II}^{-1}V,$ with
$V$ as in \eqref{MetLef.4}, also form all the sections of a smooth vector
bundle, denoted $\LT M_{\mr}.$ This vector bundle inherits a complex
structure and hence has a smooth Hermitian metric, which is unique up to a
positive smooth conformal factor on $M_{\mr}.$ The main result of this
paper is:

\begin{theorem*}\label{MetLef.5} The fiber metrics of fixed constant
curvature on a Lefschetz fibration, in the sense discussed above, extend to
a continuous Hermitian metric on $\LT M_{\mr}$ which is related to a
smooth Hermitian metric on this complex line bundle by a log-smooth
conformal factor.
\end{theorem*}

The notion of log-smoothness here, for a function, is the same as
polyhomogeneous conormality with non-negative integral powers and linear
multiplicity of slope one. Conormality in this context for
$f:M_{\mr}\longrightarrow \bbR$ can be interpreted as the `symbol
estimates' that
\begin{equation}
f\in\cA(M_{\mr})\Longleftrightarrow \Diffb*(M_{\mr})f\subset L^{\infty}(M_{\mr})
\label{MetLef.6}\end{equation}
which in fact implies that the space of these functions is stable under the
action, $\Diffb*(M_{\mr})\cA(M_{\mr})\subset\cA(M_{\mr}).$
Polyhomogeneity means the existence of appropriate expansions at the
boundary. On a manifold with boundary, $M,$ log-smoothness of a conormal
function $f\in\cA(M)$ means the existence of an expansion at the boundary,
generalizing the Taylor series of a smooth function, so for any product
decomposition near the boundary with boundary defining function $x,$ there
exist coefficients $a_{j,k}\in\CI(\pa M),$ $j\ge0,$ $j\ge k\ge0$ such that
for any finite $N,$
\begin{equation}
f-\sum\limits_{j\le N,0\le k\le j} a_{j,k}x^j(\log x)^k\in
x^N\cA([0,1)\times\pa M),\ \forall\ N.
\label{MetLef.7}\end{equation}
We denote the linear space of such functions $\CI_{\log}(M),$ it is
independent of choices.

In the case of a manifold with corners the definition may be extended by
iteration of boundary codimension. Thus $f\in\CI_{\log}(M_{\mr})$ if for
any product decompositions of $M_{\mr}$ near the two boundaries there are
corresponding coefficients $a_{j,k,b}\in\CI_{\log}(B_{b}),$ $b=\I,\II,$ such
that  
\begin{equation}
f-\sum\limits_{j\le N,0\le k\le j}a_{j,k,b}x_b^j(\log x_b)^k\in
x_{b}^N\cA([0,1)\times B_b),\ b=\I,\II,\ \forall\ N.
\label{MetLef.8}\end{equation}
There are necessarily compatibility conditions between the two expansions
at the corners, $B_{\I}\cap B_{\II},$ and together they determine $f$ up to
a smooth function on $M_{\mr}$ vanishing to infinite order on both
boundaries. In this sense the conformal factor in the main result above is
`essentially smooth'.

In the model setting, \eqref{MetResMod.3}, there is an explicit family of
fiber metrics, the `plumbing metric', of curvature $-1,$
\begin{equation}
\begin{gathered}
g_P=(\frac{\pi \log|z|}{\log|t|} \csc \frac{\pi \log|z|}{\log|t|})^2ds_0^2,\\
g_0=(\frac{|dz|}{|z|\log |z|})^2.
\end{gathered}
\label{MetResMod.4}\end{equation}
This metric can be extended (`grafted' as in \cite{MR2399166}) to give an
Hermitian metric on $\LT M_{\mr}$ which has curvature $R$ equal to
$-1$ near $B_{\II}$ and to second order at $B_{\I}.$ We prove the Theorem
above by constructing the conformal factor $e^{2f}$ for this metric which
satisfies the curvature equation, ensuring that the new metric has
curvature $-1:$
\begin{equation}
(\Lap+2)f+(R+1)=-e^{2f}+1+2f=O(f^2).
\label{MetLef.9}\end{equation}

This equation is first solved in the sense of formal power series (with logarithms)
at both boundaries, $B_{\I}$ and $B_{\II},$ which gives us an approximate
solution $f_0$ with
$$
-\Lap f_0=R+e^{2f_0}+g,\ g\in s_t^\infty\CI(\Mm).
$$ 
Then a solution $f=f_0+\tilde f$ to \eqref{MetLef.9} amounts to solving
$$
\tilde f =-(\Lap+2)^{-1}\left(2\tilde f(e^{2f_0}-1) + e^{2f_0}(e^{2\tilde
  f}-1-2\tilde f)-g\right)=K(\tilde f).
$$
Here the non-linear operator $K$ is at least quadratic in $\tilde f$ and
the boundedness of $(\Lap+2)^{-1}$ on $\rho_{\II}^{-\ha}H^M_{\bo}(\Mm)$ for
all $M$ allow the Inverse Function Theorem to be applied to show that
$\tilde f \in s_t^\infty\CI(\Mm)$ and hence that $f$ itself is log-smooth.

In \S\ref{resP} the model space and metric are analysed and in \S\ref{Glob}
the global resolution is described and the proof of the Theorem above is
outlined. The linearized model involves the inverse of $\Lap+2$ for the
Laplacian on the fibers and the uniform behavior, at the singular fibers,
of this operator is explained in \S\ref{Bounds}. The solution of the
curvature problem in formal power series is discussed in \S\ref{Formal} and
using this the regularity of the fiber metric is shown in \S\ref{CurvatureEqn}.

In \cite{Mazzeo2013} Rafe Mazzeo mentions joint work with Swoboda which
is closely related to the expansion for the metric discussed here. Our
interest in the behavior of the fiber metrics was stimulated by the
possibility, arising in discussion with Michael Singer, of extending the
work of Fine~\cite{fine2004constant}, to the Lefschetz case.

\paperbody
\section{The plumbing model}\label{resP}

We start with a description of the real resolution of the plumbing variety,
given in \eqref{MetResMod.3}, and the properties of the fiber metric,
\eqref{MetResMod.4}, on the resolved space. As noted above there are three
steps in this resoluton, first the fiber complex structure is resolved, in
a real sense, then two further steps are required to resolve the fiber
metric.

The plumbing variety itself is smooth with $z$ and $w$ global complex
coordinates -- it is the model singular fibration $\phi$ which is to be
`resolved' in the real sense. The fibers above each $t\not=0$ are annuli
\begin{equation}
\xymatrix{
\{|t|\le|z|\le\tq\}\ar[r]^{w=t/z}&\{|t|\le|w|\le\tq\}
}
\label{MetResMod.7}\end{equation}
whereas the singular fiber above $t=0$ is the union of the two discs at
$z=0$ and $w=0$ identified at their origins 
\begin{equation}
\phi^{-1}(0)=\{|z|\le\tq\}\cup\{|w|\le\tq\}/(\{z=0\}\sim\{w=0\}).
\label{MetResMod.8}\end{equation}
Note that the differential of $\phi$ vanishes at the singular point $z=w=0$
so any smooth vector field on the range which lifts under it, i.e.\ is
$\phi$-related to a smooth vector field on $P,$ vanishes at $t=0.$ Conversely,
$t\pa_t$ is $\phi$-related to both $z\pa_z$ and $w\pa_w$ whereas the vector
field
\begin{equation} V=z\pa_z-w\pa_w
\label{MetResMod.23}\end{equation}
annihilates $\phi^*t$ and so is everywhere tangent to the fibers of $\phi.$

The first step in the resolution of $\phi:P\longrightarrow \bbD_{\ha}$
consists in passing to the commutative square 
\begin{equation}
\xymatrix{
P_{\dbar}\ar[r]^-{\phi_{\dbar}}\ar[d]&[\bbD_{\ha},0]\ar[d]\\
P\ar[r]_{\phi}&\bbD_{\ha}.
}
\label{MetResMod.9}\end{equation}
Here $[\bbD_{\ha},0]$ is the space obtained by real blow up of the origin in
the disk, which can be realized globally as 
\begin{equation}
[\bbD_{\ha},0]\simeq[0,\ha]\times\bbS
\ni(r,\theta)\longmapsto t=re^{i\theta}\in\bbD_{\ha}
\label{MetResMod.10}\end{equation}
if $\bbS=\bbR/2\pi\bbZ.$ As a real blow-up $[\bbD_{\ha},0]$ is a
well-defined manifold with boundary and any diffeomorphism of $\bbD_{\ha}$
fixing the origin lifts to a global diffeomorphism. The complex structure
on $\bbD_{\ha}$ lifts to a complex structure on $\bT[\bbD_{\ha},0]$
generated by $t\pa_t=r\pa_r+i\pa_\theta$ in terms of \eqref{MetResMod.10}.

\begin{proposition}\label{MetResMod.11} The space 
\begin{equation}
P_{\dbar}=[P;\{z=0\}\cup\{w=0\}],
\label{MetResMod.12}\end{equation}
obtained by the real blow-up of the two normally-intersecting divisors
forming the singular fiber of $\phi,$ gives a commutative diagram
\eqref{MetResMod.9} in which $\phi_{\dbar}$ is a b-fibration with 
\begin{equation}
\phi_{\dbar}^*\cI_{\dbar}=\cI_{\IL}\cI_{\IR}
\label{MetResMod.13}\end{equation}
where $\cI_{\IL}$ and $\cI_{\IR}$ correspond to the two boundary components
introduced by the blow-up, forming the proper tranforms of $z=0$ and $w=0$
respectively.
\end{proposition}

\begin{proof} The two divisors forming the singular fiber $\phi^{-1}(0)$
  are each contained in a product product neighborhood
  $\bbD_{\ha}\times\bbD_{\tq}\subset P$ and 
  $\bbD_{\tq}\times\bbD_{\ha}\subset P.$ The transversality of their
  intersection is clear and it follows that the blow-up is well-defined
  independently of order with the new front faces being
\begin{equation}
B_{\IL}=\bbS\times[\bbD_{\tq},\{0\}]\subset
P_{\dbar},\ B_{\IR}=[\bbD_{\tq},\{0\}]\times\bbS\subset P_{\dbar}.
\label{MetResMod.15}\end{equation}
Here each of the blown up disks corresponds to the introduction of polar
coordinates, so $r_z=|z|$ is a defining function (globally) for $B_{\IL}$ and
$r_w=|w|$ for $B_{\IR}.$ Since $r_t=|t|$ is a defining function for the blown-up
disk in the range and 
\begin{equation}
r_t=r_zr_w
\label{MetResMod.21}\end{equation}
the b-fibration condition follows from the behaviour of the corresponding
angular variables
\begin{equation}
e^{i\theta_t}=e^{i\theta_z}e^{i\theta_w}.
\label{MetResMod.20}\end{equation}

As a compact manifold with corners, $P_{\dbar}$ is globally the product of an
embedded manifold in $\bbR^2$ and a 2-torus
\begin{equation}
P_{\dbar}=\{(r_z,r_w);0\le
r_z,r_w\le\tq,\ r_zr_w\le\ha\}\times\bbS_z\times\bbS_w.
\label{MetResMod.24}\end{equation}
\end{proof}

This first step in the resolution resolves the complex structure
in a real sense. In particular the vector fields tangent to the fibers of
$\phi_{\dbar}$ and to the boundaries form all the sections of a subbundle
of $\bT P_{\dbar}$ which has a complex structure, spanned by the lift of
the single vector field \eqref{MetResMod.23}.

Although the complex structure is effectively resolved, the plumbing metric
in \eqref{MetResMod.4} is not. That $g_P$ has curvature $-1$ on the fibers,
away from the singular point, can be seen by changing variables to $s=\log
r,$ $r=r_z$ and $\theta=\theta _z$ in terms of which
$$
g_P=(\frac{\pi/\log|t|}{\sin(\pi s/\log |t|)})^2(ds^2+d\theta^2).
$$
It then follows from the standard formula for the Gauss curvature that
$$
R=-\frac{1}{2\sqrt{fg}}(\pa_r (\frac{\pa_r g}{\sqrt{fg}})+\pa_\theta
(\frac{\pa_\theta f}{\sqrt{fg}}))=-1.
$$

In view of the coefficients in $g_P$ it is natural to introduce the inverted
logarithms of the new boundary defining functions, so replacing the
radial by the logarithmic blow-up. Thus
\begin{equation}
s_z=\ilog{r_z}=\frac1{\log\frac1{r_z}},\ s_w=\ilog{r_w}
\label{MetResMod.25}\end{equation}
become new boundary defining functions in place of $r_z$ and $r_w.$ The
space with this new \ci\ structure can be written
\begin{equation}
[P;\{z=0\}_{\log}\cup\{w=0\}_{\log}].
\label{MetResMod.26}\end{equation}

However, even after this second step, the fiber metric does not have smooth
coefficients:
$$
g_P=\frac{\pi^2s_t^2}{\sin^2(\frac{\pi s_t}{s_w})}
\left(\frac{ds_w^2}{s_w^4}+d\theta_w^2\right). 
$$
Indeed $s_t=\frac{s_zs_w}{s_z+s_w}$ is not a smooth function on the space
\eqref{MetResMod.26}.

The final part of the metric resolution is to blow up, radially, the corner
formed by the intersection of the two logarithmic boundary faces
\begin{equation}
P_{\mr}=[[P;\{z=0\}_{\log}\cup\{w=0\}_{\log}];\{s_z=s_w=0\}].
\label{MetResMod.27}\end{equation}
In terms of the presentation \eqref{MetResMod.24} this preserves
the torus factor and replaces the 2-manifold with corners by a new
one with more smooth functions and an extra boundary hypersurface.

\begin{proposition} The model Lefschetz fibration $\phi$ lifts to a
  b-fibration $\phi_{\mr}$ giving a commutative diagram
\begin{equation}
\xymatrix{
P_{\mr}\ar[r]^-{\phi_{\mr}}\ar[d]&[\bbD_{\ha};\{0\}_{\log}]\ar[d]\\
P\ar[r]_{\phi}&\bbD_{\ha}.
}
\label{MetResMod.28}
\end{equation}
\end{proposition}

\begin{proof} The radial variables on the spaces $P_{\dbar}$ and
$[\bbD_{\ha},\{0\}]$ are related by
\begin{equation}
|t|=|z||w|\Longrightarrow s_t=\frac{s_zs_w}{s_z+s_w},\ s_t=\ilog{|t|}
\end{equation}
so $\phi$ does not lift to be smooth. However, consider the further
introduction of the radial variable $R=(s_z^2+s_w^2)^{\ha}$ and the smooth
defining functions $R_z=s_z/R,$ $R_w=s_w/R$ for the lifts of the two
boundary hypersurfaces. Then
\begin{equation}
s_t=\frac{R_zRR_w}{R_z+R_w}
\end{equation}
which is smooth since $R_z$ and $R_w$ have disjoint zero sets. It follows
that $\phi$ lifts to a b-fibration as in \eqref{MetResMod.28} under which
the boundary ideal lifts to the product of the three ideals
\begin{equation}
\phi_{\mr}^*\cI_{s_t}=\cI_{R_z}\cI_R\cI_{R_w}.
\label{MetResMod.29}\end{equation}
\end{proof}

The generator $V,$ in \eqref{MetResMod.23}, of the fiber
tangent space of $\phi$ lifts to $P_{\dbar}$ as
$$
V=r_z\pa_{r_z} - r_w\pa_{r_w}- i\pa_{\theta_z} +i \pa_{\theta_w}
$$
in terms of the coordinates in \eqref{MetResMod.20} and
\eqref{MetResMod.21}. Under the introduction of the logarithmic
variables in \eqref{MetResMod.25} it further lifts to 
\begin{equation*}
V=s_z^2 \pa_{s_z} -s_w^2\pa_{s_w} - i\pa_{\theta_z} +i \pa_{\theta_w}.
\label{MetLef.25}\end{equation*}
In a neighborhood of the the lift of the face $s_z=0$ to $P_{\mr}$ the
variables $s_w$ (defining the new front face) and $\rho_z=s_z/s_w\in[0,\infty)$
(defining the lift of $s_z=0)$ are valid and
\begin{equation}
V=-s_w(s_w\pa_{s_w}-\rho_z \pa_{\rho_z}-\rho_z^2 \pa_{\rho_z})- i\pa_{\theta_z}
+i \pa_{\theta_w}.
\label{MetResMod.31}\end{equation}

Reviewing the three steps in the construction of $P_{\mr},$ notice that
the two holomorphic defining functions $z$ and $w$ are well-defined up to
constant multiples and addition of (holomorphic) terms $O(|z|^2)$ and
$O(|w|^2)$ respectively. Under these two changes, the logarithmic variables
$s_z$ change to $s_z+s_z^2G$ with $G\in\CI(P_{\mr})$ smooth. The same is
true of $s_w$ so it follows that the radial variable 
\begin{equation}
R=(s_z^2+s_w^2)^{1/2}\in\CI(P_{\mr}),
\label{MetLef.10}\end{equation}
which defines the front face, is also uniquely defined up to an additive term
vanishing quadratically there. This determines a `cusp' structure at
$B_{\II}$ and from \eqref{MetResMod.31} we conclude that

\begin{lemma}\label{MetLef.11} The vector field $R^{-1}V$ on $P_{\mr}$
  spans a smooth complex line bundle, $\LT P_{\mr}$ over $P_{\mr}$ with
  underlying real plane bundle having smooth sections precisely of the form
  $R^{-1}W$ where $W$ is a smooth vector field tangent to the boundaries,
  to the fibers of $\phi_{\mr}$ and satisfying $WR=O(R^2)$ at $R=0.$ 
\end{lemma}

It is natural to consider this bundle, precisely because

\begin{lemma}\label{MetLef.12} The plumbing metric defines an Hermitian
  metric on $\LT P_{\mr}.$
\end{lemma}

\begin{proof} 
On $P_{\mr},$ in a neighborhood of the lift of $\{s_z=0\}$ as discussed above,
$$
s_t=\ilog{|t|}=\frac{s_zs_w}{s_z+s_w}=\frac{\rho_z s_w}{1+\rho_z},\
\frac{\log|z|}{\log|t|}=\frac{1}{1+\rho_z}
$$
so the fiber metric lifts to
\begin{equation}
g=\frac{\pi^2s_t^2}{\sin^2(\frac{\pi s_t}{s_w})}
\left(\frac{ds_w^2}{s_w^4}+d\theta_w^2\right) 
=\frac{\pi^2s_t^2}{\sin^2(\frac{\pi}{1+\rho_z})}
\left(\frac{d\rho_z^2}{s_t^2(1+\rho_z)^4}+d\theta_z^2\right).
\label{MetResMod.30}\end{equation}
This is Hermitian and the length of $V$ relative to it is a smooth positive
multiple of $R^2.$
\end{proof}

\section{Global resolution and outline}\label{Glob}

It is now straightforward to extend the resolution of the plumbing variety
to a global resolution of any Lefschetz fibration as outlined in the
Introduction. By hypothesis, the singular fibers of a Lefschetz fibration
$\psi,$ as in \eqref{MetLef.2}, are isolated and each contains precisely
one singular point. Near the singular point the map $\psi$ is reduced to
$\phi$ by local complex diffeomorphisms. Thus each singular fiber is a connected 
compact real manifold of dimension two with a trasversal self-intersection. The
real blow-up of such a submanifold is well-defined, since it is locally
well-defined away from the self-intersection and well-defined near the
intersection in view of the transversality. Thus  
\begin{equation}
M_{\dbar}=[M,\phi^{-1}(S)]\overset{\psi_{\dbar}}\longrightarrow [Z,S]
\label{MetLef.13}\end{equation}
reduces to $\phi_{\dbar}$ near the preimage of the finite singular set
$F\subset M.$ Similarly, the logarithmic step can be extended globally
since away from the singular set it corresponds to replacing $|z|,$ by
$\ilog|z|.$ Here $z$ is a local complex defining function with
holomorphic differential along the singular fiber. Finally, the third step
is within the preimage of the set of the singular points and so is precisely
the same as for the plumbing variety.

Thus the resolved space $\Mm$ with its global b-fibration
\eqref{MetLef.3} is well-defined as is the Hermitian bundle $\LT \Mm$
which reduces to $\LT P_{\mr}$ near the singular points and is otherwise
the bundle of fiber tangents to $\Mm$ with its inherited complex structure.

To arrive at the description of the constant curvature fiber metric, as an
Hermitian metric on $\LT \Mm$ we start with the ``grafting''
construction of Obitsu and Wolpert which we interpret as giving a good initial
choice of Hermitian metric. Namely choose any smooth Hermitian metric $h_0$
on $\LT \Mm;$ from Lemma~\ref{MetLef.12} 
\begin{equation}
g_{\Pl}=e^{f_{\Pl}}h_0\text{ near }B_{\II},\ f_{\Pl}\text{ smooth.}
\label{MetLef.14}\end{equation}

Away from the singular set, near the singular fiber, $\psi$ is a fibration
in the real sense. Thus, it has a product decomposition, with the fibration
$\psi$ the projection, and this can be chosen to be consistent with the
product structure on $P$ away from the singular point. Then the complex
structure on the fibers is given by a smoothly varying tensor $J.$ The
constant curvature metric $g_0$ on the resolved singular fiber may
therefore be extended trivially to a metric on the fibers nearby, away from
the singular points. This has non-Hermitian part vanishing at the singular
fiber, so removing this gives a smooth family of Hermitian metrics reducing
to $g_0$ and so with curvature equal to $-1$ at the singular fiber. After
blow-up this remains true since the regular part of the singular fiber is
replaced by a trivial circle bundle over it. On the introduction of the
logarithmic variables in the base and total space, the curvature of this
smooth family, $g_\I,$ is constant to infinite order at the singular fiber
since it is equal to the limiting metric $g_0$ to infinite order. Comparing
$g_\I$ to the chosen Hermitian metric gives a conformal factor
$g_\I=e^{f_\I}h,\ f_{\I}\in\CI(N)$ where $N$ is a neighborhood of $B_{\I}$
excluding a neighborhood of $B_{\II}.$ Moreover, $g_{\Pl}$ is also equal to
the trivial extension of $g_0$ to second order in a compatible
trivialization so the two conformal factors 
\begin{equation}
f_{\I}=f_{\Pl}\text{ to second order}
\label{MetLef.16}\end{equation}
in their common domain of definition.

The grafting construction of Obitsu and Wolpert interpreted in this setting
is then to choose a cutoff $\chi\in\CI(\Mm)$ equal to $1$ in a
neighborhood of $B_{\II}$ and supported near it and to set 
\begin{equation}
h=e^{\chi f_{\Pl}+(1-\chi)f_{\I}}h_0.
\label{MetLef.17}\end{equation}
It follows from the discussion above that $h$ is a smooth Hermitian metric
on $\LT \Mm$ near the preimage of the singular fibers and that its curvature 
\begin{equation}
R(h)=\begin{cases}
-1\text{ near }B_{\II}\\
-1+O(s_t^2)\text{ near }B_{\I}.
\end{cases}
\label{MetLef.18}\end{equation}
We therefore use this in place of the initial choice of Hermitian metric.

Let $g$ be the unique Hermitian constant curvature metric on the regular
fibers of $\psi,$ so $g=e^{2f}h.$ The curvatures are related by 
$$
R(g)e^{2f}=\Lap_{h}f+R(h),
$$
which reduces to the curvature equation 
\begin{equation}
\Lap f+R(h)=-e^{2f},\ \Lap=\Lap_h.
\label{MetLef.19}\end{equation}

The linearization of this equation is 
\begin{equation}
(\Lap +2)f=-(R(h)+1).
\label{MetLef.20}\end{equation}
The uniform invertibility of $\Lap+2$ with respect to the metric $L^2$
norm, shown below, implies that \eqref{MetLef.19} has a unique small
solution for small values of the parameter. The proof of the Theorem in the
Introduction therefore reduces to the statement that \eqref{MetLef.19} has
a log-smooth solution vanishing at the boundary.

\section{Bounds on $(\Lap+2)^{-1}$}\label{Bounds}

In the linearization of the curvature equation \eqref{MetLef.20}, the
operator $\Lap+2,$ for the fixed initial choice of smooth fiber hermitian
metric, appears. For the Laplacian on a compact manifold, $\Lap+2$ is an
isomorphism of any Sobolev space $H^{k+1}$ to $H^{k-1},$ in particular this
is the case for the map from the Dirichlet space to its dual, corresponding
to the case $k=0.$ For a smooth family of metrics on a fibration the family
of Dirichlet spaces forms the fiber $H^1$ space and its dual the fiber
$H^{-1}$ space and $\Lap+2$ is again an isomorphism between them. These
spaces are modules over the $\CI$ functions of the total space and this, plus a simple commutation
argument, shows that in this case of a fibration $\Lap+2$ is an isomorphism
for any $k\ge1$ between the space of functions with up to $k$ derivatives,
in all directions, in the Dirichlet domain to the space with up to $k$
derivatives in the dual to the Dirichlet space. In particular it follows
from this that $\Lap+2$ is an isomorphism on functions supported away from
the boundary: 
\begin{equation}
\Lap+2:\CIc({M_{\reg}})\longleftrightarrow
\CIc(M_{\reg}),\ M_{\reg}=\Mm\setminus\pa \Mm.
\label{MetLef.33}\end{equation}
We extend this result up to the boundary of the resolved space for the
Lefschetz fibration in terms of tangential regularity.

\begin{proposition}\label{MetLef.43} For the Laplacian of the grafted metric
\begin{equation}
(\Lap+2)^{-1}:\rho _{\II}^{-\ha}H^k_{\bo}(\Mm)\longrightarrow \rho
  _{\II}^{-\ha}H^k_{\bo}(\Mm)\ \forall\ k\in\bbN.
\label{MetLef.42}\end{equation}
\end{proposition}
\noindent The main complication in the proof arises from
the fact that the Dirichlet space is not a \ci\ module.

First consider the following analog of Fubini's theorem.

\begin{lemma}\label{MetLef.22} For the fiber metrics corresponding to an
  Hermitian metric on $\LT \Mm,$ the metric density is of the form
\begin{equation}
|dg|=\rho_{\II}\nu_{\bo,\fib}
\label{MetLef.21}\end{equation}
and the space of weighted $L^2$ functions with values in the $L^2$ spaces
of the fibers can be realized as
\begin{equation}
L^2(\Mm;|dg|\phi_{\mr}^*\nu_{\bo}(Z_{\mr}))=
L^2_{\bo}(Z_{\mr};L^2(|dg|))=\rho _{\II}^{-\ha}L^2_{\bo}(\Mm).
\label{MetResMod.44}\end{equation}
\end{lemma}

\begin{proof} Away from $B_{\II}\subset \Mm$ the resolved map
  $\psi_{\mr}$ is a fibration, $\LT \Mm$ is the fiber tangent bundle
  and the boundary is in the base. Thus \eqref{MetLef.21} and
  \eqref{MetResMod.44} reduce to the local product decomposition for a
  fibration and Fubini's Theorem.

It therefore suffices to localize near $B_{\II}$ and to consider the plumbing
metric since all hermitian metrics on $\LT \Mm$ are
quasi-conformal. The symmetry in $z$ and $w$ means that it suffices to
consider the region in which $\rho _z=s_z/s_w$ and $s_w$ are defining
functions for the two boundary hypersurfaces $B_{\I}$ and $B_{\II}$
respectively. The plumbing metric may then be written
$$
  g=\frac{\pi^2s_t^2}{\sin^2(\frac{\pi s_t}{s_w})}
\left(\frac{ds_w^2}{s_w^4}+d\theta_w^2\right) 
=\frac{\pi^2s_t^2}{\sin^2(\frac{\pi}{1+\rho_z})}
\left(\frac{d\rho_z^2}{s_t^2(1+\rho_z)^4}+d\theta_z^2\right).
$$
Thus the fiber area form,
$$
|dg|=\frac{\pi^2s_t^2}{\sin^2(\frac{\pi}{1+\rho_z})}\frac{d\rho_z}{s_t
  (1+\rho_z)^2 d\theta_z}=f(\rho_z)\frac{s_t}{\rho_z} \frac{d
  \rho_z}{\rho_z} d\theta_z = \tilde f(\rho_z) s_w \frac{d \rho_z}{\rho_z}
d\theta_z,
$$
is a positive multiple of $s_w\frac{d\rho _z}{\rho_z}d\theta_z$ from which
\eqref{MetLef.21} follows.

The identication \eqref{MetResMod.44} holds after localization away from
$B_{\II}$ and locally near it
\begin{equation*}
||f||^2_{L^2_b(Z_{\mr});L^2(dg))}=
\int \int|f|^2 |dg|\frac{ds_t}{s_t}d\theta_t
=\int_{Z_{\mr}}|f|^2 \rho _{\II}\nu_{\bo}.
\label{MetLef.24}\end{equation*}
\end{proof}

Since $(\Lap+2)^{-1}$ is a well-defined bounded operator on the metric
$L^2$ space which depends continuously on the parameter in $Z\setminus S$
with norm bounded by $1/2,$ it follows from \eqref{MetResMod.44} that 
\begin{equation}
(\Lap+2)^{-1}\text{ is bounded on }\rho _{\II}^{-\ha}L^2_{\bo}(\Mm).
\label{MetLef.34}\end{equation}

We consider the `total' Dirichlet space based on this $L^2$ space -- we are
free to choose the weighting in the parameter space. Thus, let $D$ be the
the completion of the smooth functions on $\Mm$ supported in the
interior with respect to

\begin{equation}
\|u\|^2_{D}=\int\left(|d_{\fib}u|_g^2+2|u|^2\right)|dg|\phi^*\nu_{\bo}(Z_{\mr}).
\label{MetLef.36}\end{equation}
Note that $D$ depends only on the quasi-isometry class of the fiber
Hermitian metric but does depend on the induced fibration of the boundary $B_{\II}.$

The dual space, $D',$ to $D$ as an abstract Hilbert space, may be embedded
in the extendible distributions on $\Mm$ using the volume form
$\phi_{\mr}^*\nu_{\bo}|dg|.$ As is clear from the discussion below, the
image is independent of the choice of, $\nu_{\bo},$ of a logarithmic area
form on $Z_{\mr}$ but the embedding itself depends on this choice. Thus,
$\tilde v\in D'$ is identified as a map $v:\dCIc(\Mm)\longrightarrow \bbC$ by 
\begin{equation}
\int v\phi |dg|\phi^*\nu_{\bo}(Z_{\mr})=\tilde v(\phi).
\label{MetLef.27}\end{equation}

We consider the space of vector fields $\mathcal{W}\subset\rho
_{\II}^{-1}\bV(\Mm)$ which are tangent to the fibers of $\psi_{\mr}$
and to the fibers of $B_{\II}$ and which commute with $\pa_{\theta_z}$ and
$\pa_{\theta_w}$ near $B_{\II}.$

\begin{proposition} For the grafted metric
$$
\Lap+2: D\rightarrow D'\subset \CmI(\Mm)
$$
is an isomorphism, where the elements of $D'$ are precisely those
extendible distributions which may be written as finite sums
\begin{equation}
v=\sum\limits_{j}W_ju_j,\ W_j\in\mathcal{W}, \ u_j\in \rho _{\II}^{-\ha}L^2_{\bo}(Z_{\mr})
\label{MetLef.28}\end{equation}
and has the injectivity property that 
\begin{equation}
u\in\CmI(\Mm),\ (\Lap+2)u\in D'\Longrightarrow u\in D.
\label{MetLef.35}\end{equation}
\end{proposition}

\noindent This result remains true for any Hermitian
  metric on $\LT \Mm$ but is only needed here for the grafted metric
  which is equal to the plumbing metric near $B_{\II}.$

\begin{proof} Although defined above by completion of the space of
smooth functions supported away from the boundary of $\Mm$ with
respect to the norm \eqref{MetLef.36} the space $D$ can be identified in
the usual way with the subspace of $\CmI(\Mm)$ consisting of those 
\begin{equation}
u\in \rho _{\II}^{-\ha}L^2_{\bo}(\Mm)\Mst \mathcal{W}\cdot u\subset
\rho _{\II}^{-\ha}L^2_{\bo}(\Mm)
\label{MetLef.37}\end{equation}
with the derivatives taken in the sense of extendible
distributions. Indeed, choosing a cutoff $\mu\in\CIc(\bbR)$ which is equal
to $1$ near $0$ the sequence of multiplication operators
$1-\mu(n\rho_{\II})$ tends strongly to the identity on $\rho
_{\II}^{-\ha}L^2_{\bo}(\Mm).$ By assumption this commutes with the
elements of $\mathcal{W}$ and it follows that elements with support in the
interior of $\Mm,$ where $\psi_{\mr}$ is a fibration, are dense in
$D;$ for these approximation by smooth elements is standard.

That $\Lap+2:D\longrightarrow D'\subset\CmI(\Mm)$ is the explicit form
of the Riesz representation theorem in this setting. Then the
identification, \eqref{MetLef.28}, of elements of $D'$ follows from the
form of $\Lap.$ Away from $B_{\II},$ $D$ is a \ci\ module (since the
elements of $\cW$ are smooth there) and then \eqref{MetLef.28} is the
identification of the fiber $H^{-1}$ space. Near $B_{\II}$ we may use the
explicit form of the Laplacian for the plumbing metric.

Indeed, the local version of the Dirichlet form is
\begin{equation}
D(\phi,\psi)=\int\left( V_{\Re}\phi\overline{V_{\Re}\phi}+V_{\Im}\phi\overline{V_{\Im}\phi}
\right)\frac{ds_wd\theta_w}{s_w^2}
\label{MetResMod.62}\end{equation}
where $V$ is given by \eqref{MetResMod.31} and it follows that the
Laplacian acting on functions on the fibers can be written
\begin{equation}
\Lap=-\frac{\sin^2(\frac{\pi}{1+\rho_z})}{\pi^2s_t^2}
\left(V_{\bbR}^2+(\pa_{\theta_z}-\pa_{\theta_w})^2\right)
\label{MetResMod.67}\end{equation}
in the coordinates $s_w,$ $\rho _z,$ $\theta_w$ and $\theta_z.$

The vector fields $V_{\bbR}$ and $\rho
^{-1}_{\II}(\pa_{\theta_z}-\pa_{\theta_w})$ generate $\cW$ near $B_{\II}$
over the functions which are constant in $\theta_w$ and $\theta_z.$ If we
write $\Diff k_{\cW}(\Mm)$ for the differential operators which can be
written as sums of products of elements of at most $k$ elements of $\cW$
with smooth coefficients which are independent of the angular variables
near $B_{\II}$ then
\begin{equation}
\Lap\in\Diff2_{\cW}(\Mm).
\label{MetLef.39}\end{equation}
Moreover
\begin{equation}
\begin{gathered}
\Diff1_{\cW}(\Mm):D\longrightarrow \rho _{\II}^{-\ha}L^2_{\bo}(\Mm)\Mand\\
\Diff1_{\cW}(\Mm):\rho _{\II}^{-\ha}L^2_{\bo}(\Mm)\longrightarrow D'
\end{gathered}
\label{MetLef.40}\end{equation}
where the second statement follows by duality from the first. Together
\eqref{MetLef.39} and \eqref{MetLef.40} imply \eqref{MetLef.28}.
\end{proof}

Consider the space $\mathcal{U}\subset\bV(\Mm),$ defined analogously to
$\mathcal{W},$ as consisting of the vector fields which commute with $\pa_{\theta_z}$ and
$\pa_{\theta_w}$ near $B_{\II}.$ Then let $\Diff k_{\cU}(\Mm)$ be the
part of the enveloping algebra of $\cU$ up to order $k,$ this just consists
of the elements of $\Diffb k(\Mm)$ which commute with
$\pa_{\theta_z}$ and $\pa_{\theta_w}$ near $B_{\II}.$ We may define
`higher order' versions of the spaces $D$ and $D':$  
\begin{multline}
D_k=\{u\in D;\Diff k_{\cU}(\Mm)\cdot u\subset D\}, \\
D'_k=\{u\in D';\Diff k_{\cU}(\Mm)\cdot u\subset D'\},\ k\in\bbN.
\label{MetLef.29}\end{multline}
Since $\mathcal{U}$ spans $\bV(\Mm)$ over $\CI(\Mm)$ it follows
that 
\begin{equation}
D_k\subset \rho _{\II}^{-\ha}H^k_{\bo}(\Mm)\subset D_k'\ \forall\ k.
\label{MetLef.31}\end{equation}

\begin{proposition} For any $k,$ $\dCI(\Mm)$ is dense in $D_k$ and $D_k'$ and  
\begin{equation}
\Lap+2:D_k\longrightarrow D'_k
\label{MetLef.30}\end{equation}
is an isomorphism.
\end{proposition}

\begin{proof} The density statement follows from the same argument as for
  $D$ and $D'.$ 

Consider the commutator relation which follows directly from the definitions
\begin{equation}
[\cU,\cW]\subset \cW\Longrightarrow
[\Diff k_{\cU}(\Mm),\Lap]\subset
\Diff2_{\cW}(\Mm)\cdot\Diff{k-1}_{\cU}(\Mm),\ k\in\bbN.
\label{MetLef.41}\end{equation}

To prove \eqref{MetLef.30} we need to show that if $u\in D,$ $Q\in\Diff
k_{\cU}(\Mm)$ and $f=(\Lap+2)u\in D_k'$ then $Qu\in D.$ Assuming the
result for $Q\in\Diff{k-1}_{\cU}(\Mm)$ it follows from \eqref{MetLef.41} that
\begin{multline}
\Lap Qu=Q\Lap u+\sum\limits_{p}L_p Q_pu\Mwith L_p\in
\Diff2_{\cW}(\Mm),\ Q_p\in\Diff{k-1}_{\cU}(\Mm)\\
\Longrightarrow \Lap Qu\in D'\Longrightarrow Qu\in D
\end{multline}
by distributional uniqueness.
\end{proof}

\begin{proof}[Proof of Proposition~\ref{MetLef.43}] The boundedness
  \eqref{MetLef.42} follows directly from \eqref{MetLef.30} and \eqref{MetLef.31}.
\end{proof}

\section{Formal solution of $(\Lap+2)u=f$}\label{Formal}

In the previous section the uniform invertibility of $\Lap+2$ for the
grafted metric was established. In particular the case $k=\infty$ in
\eqref{MetLef.42} shows the invertibility on conormal functions. In this
section we solve the same equation, $(\Lap+2)u=f$ in formal power series
with logarithmic terms.

Let $\CI_{F}(\Mm)\subset\CI(\Mm)$ denote the subspace annihilated to infinte order at
$B_{\II}$ by the angular operators $D _{\theta_z}$ and $D_{\theta_w}.$

\begin{lemma}\label{MetResMod.53} The restriction, $\Lap_{\I},$ of the
  Laplacian to $B_{\I}$ satisfies
\begin{multline}
(\Lap_{\I}+2)^{-1}\left(\rho _{\II}(\log\rho _{\II})^kg_k\right)\\
=\rho
  _{\II}\sum\limits_{0\le p\le k+1}(\log\rho
  _{\II})^pu_p,\ u_p\in\CI_{F}(\Mm)\ \forall\ g_k\in\CI_{F}(\Mm).
\label{MetResMod.54}\end{multline}
\end{lemma}

\begin{proof} The fiber metric on $B_{\I}$ is a trivial family with respect
  to the product decomposition $B_{\I}=A\times\bbS$ where $A$ has the
  complete metric on the Riemann surface with cusps arising from the
  `removal' of the nodal points. The Laplacian is therefore essentially
  self-adjoint and non-negative, so $\Lap+2$ is invertible. Either from the
  form of a parameterix or by Fourier expansion near the cusps it follows
  that rapid decay in the non-zero Fourier modes (in both angular
  variables) is preserved by $(\Lap_{\I}+2)^{-1}.$ Near the boundary the
  zero Fourier mode satisfies a reduced, ordinary differential, equation
  with regular singular points and having indicial roots $1$ and $-2$ in
  terms of a defining function for the (resolved) cusps. Then
  \eqref{MetResMod.54} follows directly.
\end{proof}

\begin{lemma}\label{MetResMod.63} If $u\in\CI_{F}(\Mm)$
  then $\Lap u\in\CI_{F}(\Mm)$ restricts to $B_{\II}$ to $\widetilde{\Lap}_{\II}v,$
  $v=u\big|_{B_{\II}}$ where $\widetilde{\Lap}_{\II}$ is an ordinary
  differential operator of order $2$ elliptic in the interior with regular
  singular endpoints, with indicial roots $-1,2$ such that 
\begin{equation}
\Nul(\widetilde{\Lap}_{\II}+2)\subset \rho _{\I}^{-1}\CI(B_{\II})
\label{MetResMod.65}\end{equation}
has no smooth elements and for $h_j\in\CI_{F}(B_{\II})$
\begin{equation}
\begin{gathered}
(\widetilde{\Lap}_{\II}+2)^{-1}(\log\rho _{\I})^jh_j=\sum\limits_{0\le q\le
  j}(\log\rho _{\I})^q v_{q,j}+\rho _{\I}^2(\log\rho _{\I})^{j+1}w_j\\
\Mwith v_{q,j},\ w_j\in\CI_{F}(B_{\II}).
\end{gathered}
\label{MetResMod.66}\end{equation}
\end{lemma}

\begin{proof} The form of the Laplacian in \eqref{MetResMod.67} shows that
  the reduced operator $\widetilde{\Lap}_{\II}$ exists and after the change coordinates
on $B_{\II}$ to
\begin{equation}
\rho=\frac{1}{1+\rho_{\II}}
\end{equation}
becomes
\begin{equation}
\Lap+2=2-(\frac{\sin(\pi \rho)}{\pi \rho})^2[(\rho \pa_\rho)^2-\rho \pa_{\rho}].
\end{equation}
The indicial roots of this operator are $2$ and $-1$ and its homoeneity
shows that the null space has no logarithmic terms. The absence of smooth
elements in the null space follows by integration by parts and positivity.
\end{proof}

The problem that we need to solve at $B_{\II}$ is 
\begin{equation}
(\Lap+2)(\rho _{\II}w)=\rho _{\II}g+O(\rho _{\II}^2)\Longrightarrow
  (\widetilde{\Lap}^{(1)}_{\II}+2)(w\big|_{B_{\II}})=g\big|_{B_{\II}}.
\label{MetResMod.74}\end{equation}
Since the parameter, $s_t,$ is the product of defining functions for
$B_{\I}$ and $B_{\II}$ and commutes through the problem this can be solved
by dividing by it. Thus $\widetilde{\Lap}^{(1)}_{\II}$ is obtained
from $\widetilde{\Lap}_{\II}$ by conjugating by a boundary defining function on
$B_{\II}$ so the preceding Lemma can be applied after noting the shift of
the indicial roots.

\begin{lemma}\label{MetResMod.72} For the conjugated operator on $B_{\II},$ 
\begin{equation}
\Nul(\widetilde{\Lap}^{(1)}_{\II}+2)\subset\CI(B_{\II}) 
\label{MetResMod.76}\end{equation}
with the Dirichlet problem uniquely solvable and 
\begin{equation}
\begin{gathered}
(\widetilde{\Lap}_{\II}+2)^{-1}(\log\rho _{\I})^jh_j=\sum\limits_{0\le q\le
  j}(\log\rho _{\I})^q v_{q,j}+\rho _{\I}^3(\log\rho _{\I})^{j+1}w_j\\
\Mwith v_{q,j},\ w_j\in\CI_{F}(B_{\II}).
\end{gathered} 
\label{MetResMod.75}\end{equation}
\end{lemma}

To express the form of the expansion which occur below, consider the
space of polynomials in $\log\rho _{\I}$ and $\log\rho _{\II}$
with coefficients in $\CI_F(\Mm)$
\begin{equation}
\cP^{k}=\left\{u=\sum\limits_{0\le l+p\le k}(\log\rho
_{\I})^l(\log\rho _{\II})^pu_{l,p},\ u_{l,p}\in\CI_F(\Mm)\right\}.
\label{MetLef.46}\end{equation}
We also consider the filtration of these spaces by the maximal order in
each of the variables: 
\begin{equation}
\begin{gathered}
\cP^{k,j}_{\I}=\left\{u=\sum\limits_{0\le l+p\le k,\ l\le j}(\log\rho
_{\I})^l(\log\rho _{\II})^pu_{l,p},\ u_{l,p}\in\CI_F(\Mm)\right\}, \ j\le k\\
\cP^{k,m}_{\II}=\left\{u=\sum\limits_{0\le l+p\le k,\ p\le m}(\log\rho
_{\I})^l(\log\rho _{\II})^pu_{l,p},\ u_{l,p}\in\CI_F(\Mm)\right\},\ m\le k.
\end{gathered}
\label{MetLef.60}\end{equation}
Since the coefficients are in $\CI_F(\Mm),$ $\Lap$ acts as a smooth
b-differential operator on all of these spaces. If $u\in\cP^{k,p}_{\I},$ then
$u=u_p+u'$ with $u'\in\cP^{k,p-1}_{\I}$ and $u_p=v(\log\rho _{\I})^p$ where
  $v\in\cP^{k-p,0}_{\I}.$ Then $\Lap u=(\Lap_{\I}v)(\log\rho _{\I})^p+f',$
$f'\in\cP^{k-1,p-1}_{\I}+\rho _{\I}\cP^{k,p-1}_{\I}$ where the first error term
corresponds to at least one derivation of $(\log\rho _{\I})^p.$ Similar
statements apply to $B_{\II}$ and $\tilde\Lap_{\II}.$

As a basis for iteration, to capture the somewhat complicated behavior of
the logarthimic terms, we first consider a partial result.

\begin{proposition}\label{MetLef.100} For each $k$ 
\begin{equation}
f\in\rho_{\II}\cP^{k}+\rho _{\I}\rho_{\II}\cP^{k+1}\Longrightarrow
\ \exists\ u\in\rho_{\II}\cP^{k+1}+\rho  _{\I}^2\rho _{\II}\cP^{k+2,k+1}_{\II}
\label{MetLef.61}\end{equation}
such that
\begin{equation}
(\Lap+2)u-f\in s_t\left(\rho _{\II}\cP^{k+1}+\rho _{\I}\rho _{\II}\cP^{k+2}\right).
\label{MetResMod.73}\end{equation}
\end{proposition}

\begin{proof} We first solve on $B_{\I},$ then on $B_{\II}.$ The second
  term in $f$ in \eqref{MetLef.61} vanishes on $B_{\I}$ so the restriction
  $f_{\I}\in\rho_{\II}\cP^{k}\big|_{B_{\I}}.$ Proceeding iteratively, suppose
  $$
f\in\rho _{\II}\cP^{k,j}_{\I}+\rho _{\I}\rho_{\II}\cP^{k+1}
$$
with $j\le k$
  and consider the term of order $j$ in $\log\rho _{\I};$ this is a
  polynomial in $\log\rho _{\II}$ of degree at most $k-j$ with coefficients
  in $\rho_{\II}\CI_F(B_{\I}).$ Applying Lemma~\ref{MetResMod.53} to the
 restriction to $B_{\I}$ gives a polynomial in $\log\rho _{\II}$ of degree
  at most $k-j+1$ with coefficients in $\rho _{\II}\CI_F(B _{\I}).$
  Extending these coefficients off $B_{\I}$ and restoring the coefficient
  of $(\log\rho _{\I})^j$ gives $v_j\in \rho _{\II}\cP^{k+1,j}_\I$ such that
\begin{equation*}
(\Lap+2)v_{j}-f=-f',\ f'\in \rho _{\II}\cP^{k,j-1}_{\I}+\rho _{\I}\rho_{\II}\cP^{k+1}.
\label{MetLef.51}\end{equation*}
Here the first part of the error arises from differentiation of the factor $(\log\rho
_{\I})^j$ in $v_{j}$ at least once. If we start with $j=k$ and proceed
iteratively over decreasing $j$ this allows us to find $v\in\rho
_{\II}\cP^{k+1}$ such that 
\begin{equation}
(\Lap+2)v-f=-g\in \rho _{\I}\rho _{\II}\cP^{k+1}.
\label{MetLef.55}\end{equation}

Now we proceed similarly by solving on $B_{\II}$ using
Lemma~\ref{MetResMod.72}. So, suppose $h\in\rho _{\I}\rho_{\II}\cP^{k+1,p}_{\II},$
for $p\le k+1.$ Then the coefficient $h_p$ of
$(\log\rho _{\II})^p$ is a polynomial of degee at most $k+1-p$ in $\log\rho
_{\I}$ with coefficients in $\rho _{\I}\rho _{\II}\CI_F(\Mm).$ Conjugating
away the factor of $\rho _{\II}$ and applying Lemma~\ref{MetResMod.72} to
the restriction to $B_{\II}$ and then extending the coefficients off
$B_{\II}$ allows us to find $w_p\in \rho _{\I}\rho
_{\II}\cP^{k+1,p}_{\II}+\rho _{\I}^2\rho _{\II}\cP^{k+2,p}_{\II},$ where the second
term arises from the possible increase in multiplicity of the logarithmic
coefficient of $\rho _{\I}^2$ in the solution, satisfying
\begin{equation}\label{MetLef.90}
(\Lap+2)w_p-g=-g'+e,\ g'\in \rho _{\I}\rho _{\II}\cP^{k+1,p-1}_{\II},\
e\in\rho _{\I}\rho _{\II}^2\cP^{k+1,p}_{\II}+\rho _{\I}^2\rho _{\II}^2\cP^{k+2,p}_{\II}
\end{equation}
where the first part of the error arises from differentiation of $(\log\rho
_{\II})^p$ at least once. Starting with $p=k+1$ and iterating over
decreasing $p$ allows us to find $w\in\rho _{\I}\rho _{\II}\cP^{k+1}+\rho
_{\I}^2\rho _{\II}\cP^{k+2,k+1}$ such that 
\begin{equation}
(\Lap+2)w-g\in \rho _{\I}\rho _{\II}^2\cP^{k+1}+\rho _{\I}^2\rho _{\II}^2\cP^{k+2}.
\label{MetLef.54}\end{equation}

Combining \eqref{MetLef.55} and \eqref{MetLef.54} gives
\eqref{MetResMod.73} since $\rho _{\I}\rho _{\II}$ is a smooth multiple of $s_t.$ 
\end{proof}

Proposition~\ref{MetLef.100} allows iteration since $s_t$ commutes through $\Lap+2.$

\begin{proposition}\label{MetResMod.56} If $f\in \rho _{\II}\cP^{k}+\rho
  _{\I}\rho _{\II}\cP^{k+1}$ then $u=(\Lap+2)^{-1}f\in
  s_t^{-\epsilon}H^{\infty}_{\bo}(\Mm)$ for any $\epsilon >0,$ has a
  complete asymptotic expansion of the form
\begin{equation}
u\simeq \sum\limits_{j\ge0} s_t^ju_j,\ u_j\in\rho _{\II}\cP^{k+j}+\rho
_{\I}\rho _{\II}\cP^{k+j+1,k+j}_{\II}.
\label{MetResMod.77}\end{equation}
\end{proposition}

\begin{proof} For any $\epsilon >0,$ $g=s_t^{\epsilon}f\in
  \rho_{\II}^{-\ha}H^{\infty}_{\bo}(\Mm)$ so
  $u=s_{t}^{-\epsilon}(\Lap+2)^{-1}g$ exists by \eqref{MetLef.42}. Comparing
  $u$ to the expansion cut off at a finite point gives \eqref{MetResMod.77}.
\end{proof}

This result can itself be iterated, asymototically summed and then the
rapidly decaying remainder term again removed to show the polyhomogeneity
of the solution for an asymptotically covergent sum over terms on the right
in \eqref{MetResMod.77}.

For the solution of the curvature equation the leading term is smooth
because of the special structure of the forcing term.

\begin{lemma}\label{MetResMod.78} If $f\in\CI(\Mm)$ has support
  disjoint from $B_{\II}$ then $u=(\Lap+2)^{-1}f$ is log-smooth and has an asymptotic
  expansion of the form 
\begin{equation}
u\simeq \rho _{\II}v_{0}+\sum\limits_{k\ge1}s_t^kv_k,\ v_k\in\rho
_{\II}\cP^k+\rho_{\I}\rho _{\II}\cP^{k+1,k}_{\II}.
\label{MetResMod.79}\end{equation}
\end{lemma}
\noindent Note that log-smoothness follows from the fact that $s_t=a\rho
_{\I}\rho _{\II},$ $a\in\CI_F(\Mm)$ so each term in the expansion can be
written as a polynomial in $\rho _{\I},$ $\rho _{\I}\log\rho _{\I},$ $\rho
_{\II}$ and $\rho _{\II}\log\rho _{\II}$ of degree at least $2k.$ 

\section{Polyhomogeneity for the curvature equation}\label{CurvatureEqn}

Under a conformal change from the grafted metric $h$ with curvature $R$ to
$e^{2f}h$ the condition for the curvature of the new metric to be $-1$
given by \eqref{MetLef.9}.  To construct the canonical metrics on the
fibers we proceed, as in the linear case discussed above, to solve
\eqref{MetLef.9} in the sense of formal power series at the two boundaries
above $s_t=0$ and then, using the Implicit Function Theorem deduce that the
actual solution has this asymptotic expansion.

\begin{lemma} For the grafted metric there is a formal power series
\begin{equation}
\sum\limits_{k\ge
  2}s_t^kf_k,\ f_2\in\CI_F(\Mm),\ f_k\in\rho_{\II}\cP^{k-2}+\rho _{\I}\rho
_{\II}\cP^{k-1,k-2}_{\II},\ k\geq 3,
\label{MetLef.56}\end{equation}
solving \eqref{MetLef.9}.
\end{lemma}
\noindent The $\cP^{k}$ are defined in \eqref{MetLef.46}; in the last term
there is no factor of $(\log\rho_{\II})^{k-1}.$  

\begin{proof} Since $R+1\in s_t^2\CI(\Mm)$ is supported away from
  $B_{\II},$ Lemma~\ref{MetResMod.78} shows that $g_1=-(\Lap+2)^{-1}(R+1)$
  is of the form \eqref{MetLef.56}. We look for the formal power series
  solution of the non-linear problem as 
\begin{equation}
f\simeq\sum\limits_{k\ge1}g_k
\label{MetLef.57}\end{equation}
Inserting this sum into the equation gives
\begin{equation}
-(\Lap+2)(\sum_{i\geq 1}g_i)=\sum_{j\geq 2}\frac{2^j}{j!}(g_1+\sum_{k\geq 2}g_k)^j+1+R.
\end{equation}
For each $i\ge 2$ we fix $g_i$ by
\begin{multline}
-(\Lap+2)g_i=
\sum_{j\geq 1}\frac{2^j}{j!}(g_1+\sum_{i-1\geq k\geq 2}g_k)^j
-\frac{2^j}{j!}(g_1+\sum_{i-2\geq k\geq 2}g_k)^j\\
=g_{i-1}P_i(g_1, g_2,...g_{i-1})
\end{multline}
where $P_i$ is a formal power series in $g_1,...g_{i-1}$ without constant
term.

Proceeding by induction we claim that 
\begin{equation}
g_i\simeq\sum\limits_{j\ge 2i}s_t^{j}g_{i,j},\ g_{i,j}\in\rho
_{\II}\cP^{j-2i}+\rho_{\I}\rho _{\II}\cP^{j-2i+1,j-2i}_{\II}.
\label{MetLef.58}\end{equation}
We have already seen that this holds for $i=1$ and using the obvious
multiplicitivity properties 
\begin{equation*}
\cP^{k}\cdot\cP^{j}\subset\cP^{j+k},\ \cP^{k}\cdot\cP^{j,j-1}_{\II}\subset\cP^{j+k,j+k-1}_{\II}
\label{MetLef.63}\end{equation*}
it follows from the inductive assumption, that \eqref{MetLef.58} holds for all smaller
indices, that 
\begin{equation}
\begin{gathered}
g_{i-1}P_i(g_1, g_2,...g_{i-1})\\
\simeq s_t^{2i} \sum\limits_{k\geq 2, j \geq 2i-2} \left(\rho_{\II}\cP^{j-2i+2} + \rho_{\I} \rho_{\II}\cP_{\II}^{j-2i+3, j-2i+2}\right) 
\left(\rho_{\II}\cP^{k-2} + \rho_{\I} \rho_{\II}\cP_{\II}^{k-1, k-2} \right) \\
\simeq
\sum\limits_{k\ge 2i}s_k^jF_k,\ F_k\in \rho
_{\II}\cP^{k-2i}+\rho_{\I}\rho _{\II}\cP^{k-2i+1, k-2i}_{\II}.
\end{gathered}
\label{MetLef.59}\end{equation}
Applying Proposition~\ref{MetResMod.56} we recover the inductive hypothesis
at the next step. Then \eqref{MetLef.56} follows from \eqref{MetLef.57} and
\eqref{MetLef.58}. 
\end{proof}

Summing the formal power series solution gives a polyhomogeneous function with
\begin{equation}
-\Lap f_0=R+e^{2f_0}+g, g\in O(s_t^\infty).
\end{equation}
Now we look for the solution as a perturbation $f=f_0+\tilde f,$ so $\tilde f$ satisfies
\begin{equation}
-\Lap \tilde f=-g+e^{2f_0}(e^{2\tilde f}-1).
\label{MetLef.52}\end{equation}
which can be rewritten as
$$
\tilde f =-(\Lap+2)^{-1}\left(2\tilde f(e^{2f_0}-1)+e^{2f_0}(e^{2\tilde
  f}-1-2\tilde f)-g\right). 
$$
So consider the nonlinear operator
\begin{equation}
K: \tilde f \mapsto (\Lap+2)^{-1}\left(2\tilde
f(e^{2f_0}-1)+e^{2f_0}(e^{2\tilde f}-1-2\tilde f)-g \right) 
\end{equation}
which acts on $s_t^{N}H_b^{M}(\Mm)$ for all $N\ge1$ and $M>2.$ Note that
for $M>2$, the b-space $H_b^{M}(\Mm)$ is closed under multiplication,
therefore this weighted Sobolev space is also an algebra. Since the nonlinear
terms are at least  quadratic, $K$ is well-defined on this domain. The
solution to \eqref{MetLef.52} satisfies $\tilde f=K(\tilde f).$

\begin{proposition}\label{MetLef.53} For any $M>1$ and $N\geq 1$ there is a
  unique solution $\tilde f \in  s_t^{N}H_b^{M}(\Mm)$ to the equation
  \eqref{MetLef.52}.\end{proposition}

\begin{proof}
We construct the solution $\tilde f$ by iteration. Let $\tilde f =
s_t^N\sum_{i\geq 2} {s_t}^i f_i$, put it into equation \eqref{MetLef.52},
divide by the common factor $s_t^N$ on both sides and then we get
\begin{equation}\label{MetLef.99}
\sum\limits_{i\geq 2} s_t^i f_i=K(\sum s_t^i f_i)=
(\Lap+2)^{-1}\left( (e^{2f_0}-1) \sum s_t^i f_i + s_t^N (\sum s_t^i f_i)^2 +s_t^{-N}g
\right)
\end{equation}
The right hand side belongs to $(\Lap+2)^{-1}(O(s_t^2))$ because of the
quadratic structure and the fact that $e^{2f_0}-1 \in O(s_t^2)$. Therefore the right hand side is the form $(\Lap+2)^{-1}(s_t h)$ where $s_t h \in \rho_{\II}^{-\ha}H^{M}_{\bo}(\Mm)$ so this quantity is well-defined using Proposition~\ref{MetLef.43}.

Now we proceed by induction. Assume that the first k terms in the expansion
have been solved, then the equation for the next term $f_k$ is given by
$$
f_k=(\Lap+2)^{-1}\left((e^{2f_0}-1) f_{k-2} + s_t^N Q(f_0, ...f_{k-1})  \right).
$$
where the polynomial $Q$ on the right hand side is a quadratic polynomial
of order $k-N.$ By using the invertibility property in Proposition~\ref{MetLef.43}, we
can now solve $f_{k}$. Therefore the induction gives us the total expansion for $\tilde f$.
\end{proof}

\begin{proof}[Proof of Theorem~\ref{MetLef.5}] From
  Proposition~\ref{MetLef.53} we obtain the solution, $f=f_0+\tilde f$, to
  the curvature equation $R(e^{2f}h)=-1.$ Since $f_0$ is the formal power
  series and $\tilde f \in s_t^\infty \CI(\Mm)$, we get the solution with
  required regularity.
\end{proof}

\providecommand{\bysame}{\leavevmode\hbox to3em{\hrulefill}\thinspace}
\providecommand{\MR}{\relax\ifhmode\unskip\space\fi MR }
\providecommand{\MRhref}[2]{%
  \href{http://www.ams.org/mathscinet-getitem?mr=#1}{#2}
}
\providecommand{\href}[2]{#2}

\end{document}